\documentclass[12pt]{article}
\usepackage{a4wide}

\usepackage[a4paper, margin=2.3cm]{geometry}
\usepackage{amsmath,amssymb,amsthm}
\usepackage{color}
\usepackage{parskip}
\usepackage{hyperref}
\usepackage{setspace}

\newtheorem{theorem}{Theorem}[section]

\newtheorem{lemma}[theorem]{Lemma}

\newtheorem{question}[theorem]{Problem}
\newtheorem*{definition*}{Definition}

\def\C{\mathcal{C}}
\def\A{\mathcal{A}}
\def\B{\mathcal{B}}
\def\D{\mathcal{D}}
\def\E{\mathcal{E}}

\begin{document}

\title{On the structure of distance sets over prime fields}\author{
    Thang Pham\thanks{Department of Mathematics, University of California at San Diego, La Jolla, CA, 92093 USA. Supported by Swiss National Science Foundation grant P2ELP2-175050.
    Email: {\tt v9pham@ucsd.edu}}\and Andrew Suk\thanks{Department of Mathematics, University of California at San Diego, La Jolla, CA, 92093 USA. Supported by an NSF CAREER award and an Alfred Sloan Fellowship. Email: {\tt asuk@ucsd.edu}}}
\maketitle
\date{}
\begin{abstract}
Let $\mathbb{F}_q$ be a finite field of order $q$ and $\E$ be a set in $\mathbb{F}_q^d$.  The \emph{distance set} of $\E$, denoted by $\Delta(\E)$, is the set of distinct distances determined by the pairs of points in $\E$. Very recently, Iosevich, Koh, and Parshall (2018) proved that if $|\E|\gg q^{d/2}$, then the \emph{quotient set} of $\Delta(\E)$ satisfies
\[\left\vert\frac{\Delta(\E)}{\Delta(\E)}\right\vert=\left\vert \left\lbrace\frac{a}{b}\colon a, b\in \Delta(\E), b\ne 0\right\rbrace\right\vert\gg q.\] In this paper, we break the exponent $d/2$ when $\E$ is a Cartesian product of sets over a prime field. More precisely, let $p$ be a prime and $A\subset \mathbb{F}_p$.  If $\E=A^d\subset \mathbb{F}_p^d$ and $|\E|\gg p^{\frac{d}{2}-\varepsilon}$ for some $\varepsilon>0$, then we have
\[\left\vert\frac{\Delta(\E)}{\Delta(\E)}\right\vert, ~\left\vert \Delta(\E)\cdot \Delta(\E)\right\vert \gg p.\]
Such improvements are not possible over arbitrary finite fields.  These results give us a better understanding about the structure of distance sets and the Erd\H{o}s-Falconer distance conjecture over finite fields.
\end{abstract}
\section{Introduction}

Let $q$ be an odd prime power, and $\mathbb{F}_q$ be the finite field of order $q$. For any two points $\mathbf{x}=(x_1, \ldots, x_d)$ and $\mathbf{y}=(y_1, \ldots, y_d)$ in $\mathbb{F}_q^d$, the distance between them is defined by
\[||\mathbf{x}-\mathbf{y}||=(x_1-y_1)^2+\cdots+(x_d-y_d)^2.\]
This function is not a norm, but it is invariant under translations, rotations, and reflections.  Given a set $\E \subset \mathbb{F}_q^d$, we define the \emph{distance set}\[\Delta(\E):=\{||\mathbf{x}-\mathbf{y}||\colon \mathbf{x}, \mathbf{y}\in \E\}.\]
The finite field variant of the Erd\H{o}s distinct distances problem was first studied by Bourgain, Katz, and Tao in \cite{bourgain-katz-tao}, who proved the following theorem.
\bigskip
\begin{theorem} [\textbf{Bourgain-Katz-Tao}, \cite{bourgain-katz-tao}]\label{BKTeq}
Suppose $q\equiv 3\mod 4$ is a prime. Let  $\E$ be a set in $\mathbb{F}_q^2$. If  $|\E|=q^{\alpha}$ with $0<\alpha<2$, then we have
\[|\Delta(\E)|\gg |\E|^{\frac{1}{2}+\varepsilon},\]
for some positive $\varepsilon=\varepsilon(\alpha)>0$.
\end{theorem}
Throughout this paper, we write $X \gg Y$ if there is a positive constant $C$ such that $X\ge CY$, and $X \ll Y$ if $Y\gg X$.

Iosevich and Rudnev \cite{io} observed that the conclusion of Theorem \ref{BKTeq} can not be extended to arbitrary finite fields in general. For instance, when $q$ is a square, i.e. $q=p^2 $ for some prime $p$, we can choose $\E=\mathbb F_p \times \mathbb F_p.$ One can check that in this case, we have $|\Delta(\E)|=|\E|^{1/2}$. Furthermore, if $-1$ is a square number in $\mathbb{F}_q$, i.e. $-1=i^2$ for some $i\in \mathbb{F}_q$,  then we can  choose $\E=\{(t, it)\in \mathbb F_q^2: t\in \mathbb F_q\}$. This set only gives us the distance zero. In light of these constructions, Iosevich and Rudnev \cite{io} made the following reformulation of the distinct distances problem, in the spirit of the Falconer distance conjecture \cite{fal}.\footnote{The Falconer distance conjecture states that for any compact set $\E\subset\mathbb{R}^d$ with the Hausdorff dimension greater than $d/2$, the distance set $\Delta(\E)$ has positive Lebesgue measure.}

\bigskip
\begin{question}\label{pr}
Let $\E$ be a set in $\mathbb{F}_q^d$, and $\Delta(\E)$ be the set of distinct distances determined by the pairs of points in $\E$. How large does $\E$ need to be to guarantee that $|\Delta(\E)|\gg q?$
\end{question}
This problem is now known as the Erd\H{o}s-Falconer distance problem over finite fields. Using Fourier methods, Iosevich and Rudnev \cite{io} proved that if $|\E|\gg q^{(d+1)/2}$, then the distance set $\Delta(\E)$ covers a positive proportion of all elements in $\mathbb{F}_q$, that is, $|\Delta(\E)| \gg q$. Hart et al. \cite{har} showed that we can have all distances whenever $|\E|\ge 4q^{\frac{d+1}{2}}$. They also gave constructions for the sharpness of the exponent $(d+1)/2$ in odd dimensions. However, in even dimensions, it is still possible to break the $(d+1)/2$ exponent. Chapman et al. \cite{chap} made the first step in this direction by showing that if $d=2$, then the exponent $3/2$ can be decreased to $4/3$, which is directly in line with Wolff's result \cite{37} for the Falconer distance problem in $\mathbb{R}^2$. It has been conjectured that in even dimensions, the assumption $|\E| \gg q^{\frac{d}{2}}$ is sufficient for $|\Delta(\E)|\gg q$.

In a recent work, Iosevich, Koh, and Parshall \cite{ikp} proved that the exponent $d/2$ holds for the \emph{quotient set} of the distance set, which is defined by
\[\frac{\Delta(\E)}{\Delta(\E)}=\left\lbrace \frac{a}{b}\colon a, b\in \Delta(\E), b\ne 0\right\rbrace.\]
The statement of their result is as follows.
\bigskip
\begin{theorem}[\textbf{Iosevich-Koh-Parshall}, \cite{ikp}]\label{IKP}
Let $\mathbb{F}_q$ be a finite field of order $q$, and $\E$ be a set in $\mathbb{F}_q^d$.
\begin{itemize}
\item[1.] If $d\ge 2$ is even and $|\E|\ge 9q^{d/2}$, then we have $$\frac{\Delta(\E)}{\Delta(\E)}=\mathbb{F}_q.$$
\item[2.] If $d\ge 3$ is odd and $|\E|\ge 6q^{d/2}$, then we have
\[\{0\}\cup \mathbb{F}_q^+\subset \frac{\Delta(\E)}{\Delta(\E)},\]
where $\mathbb{F}_q^+=\left\lbrace x^2\colon x\in \mathbb{F}_q, x\ne 0\right\rbrace$.
\end{itemize}
\end{theorem}
\bigskip
Notice that the condition $|\E|\gg q^{d/2}$ in Theorem \ref{IKP} is sharp over arbitrary finite fields, even if we wish to cover only a positive proportion of all elements in $\mathbb{F}_q$. Indeed, suppose that $q=p^2$ for some prime $p$.  By setting $\E=\mathbb{F}_p^d$, we have $|\E|=q^{\frac{d}{2}}$ and $|\Delta(\E)/\Delta(\E)|=|\mathbb{F}_p|=q^{1/2}$. We refer the interested reader to \cite{ikp} for more discussions.

Let us also remark that it seems difficult apply the methods in \cite{ikp} to the analogous problem of having the \emph{product set} of the distance set cover a positive proportion of $\mathbb{F}_q$.  Using a different approach, Iosevich and Koh \cite{ik} proved that for $\E\subset \mathbb{F}_q^d$, if $|\E|\gg q^{\frac{d}{2}+\frac{1}{4}}$, then
\[\Delta(\E)\cdot \Delta(\E)=\left\lbrace a\cdot b\colon a, b\in \Delta(\E)\right\rbrace=\mathbb{F}_q.\]
The main purpose of this paper is to show that if $\E$ is a Cartesian product of sets over a prime field $\mathbb{F}_p$, we can break the exponent $d/2$ and still guarantee that
\[\left\vert\frac{\Delta(\E)}{\Delta(\E)}\right\vert, ~\left\vert \Delta(\E)\cdot \Delta(\E)\right\vert \gg p.\]
Our first two results are for the case of the quotient set, in even and odd dimensions.
\bigskip
\begin{theorem}\label{thm1}
Let $\mathbb{F}_p$ be a prime field, and $A \subset\mathbb{F}_p$.  Then for $\E=A^{d} \subset \mathbb{F}_p^{d}$ with $d=2k$, $k\ge 2\in \mathbb{N}$, we have
\[\left\vert\frac{\Delta(\E)}{\Delta(\E)}\right\vert=\left\vert\left\lbrace \frac{a}{b}\colon a, b\in \Delta(\E), b\ne 0\right\rbrace\right\vert\ge \frac{p}{3},\]
whenever $|\E|\gg p^{\frac{d}{2}-\varepsilon}$ with $\varepsilon=\frac{d}{2}\cdot \frac{2^k-2^{k-1}-1}{2^{k}-1}$.
\end{theorem}
\bigskip
\begin{theorem}\label{thm2}
Let $\mathbb{F}_p$ be a prime field, and $A \subset\mathbb{F}_p$.  Then for $\E=A^{d} \subset \mathbb{F}_p^{d}$ with $d=2k + 1$, $k\ge 2\in \mathbb{N}$, we have\[\left\vert\frac{\Delta(\E)}{\Delta(\E)}\right\vert=\left\vert \left\lbrace\frac{a}{b}\colon a, b\in \Delta(\E), b\ne 0\right\rbrace\right\vert\ge \frac{p}{3},\]
whenever $|\E|\gg p^{\frac{d}{2}-\varepsilon}$ with $\varepsilon=d\cdot \frac{2^{k+2}-2^{k+1}-3}{2^{k+3}-6}$.
\end{theorem}
\bigskip

\bigskip
Our next two theorems are for the case of the product set, in even and odd dimensions.
\bigskip
\begin{theorem}\label{thm1'}
Let $\mathbb{F}_p$ be a prime field, and $A \subset\mathbb{F}_p$.  Then for $\E=A^{d} \subset \mathbb{F}_p^{d}$ with $d=2k$, $k\ge 2\in \mathbb{N}$, we have
\[\left\vert \Delta(\E)\cdot \Delta(\E)\right\vert=\left\vert\left\lbrace a\cdot b\colon a, b\in \Delta(\E)\right\rbrace\right\vert\gg p,\]
whenever $|\E|\gg p^{\frac{d}{2}-\varepsilon}$ with $\varepsilon=\frac{d}{2}\cdot \frac{2^{k+1}-5}{5\cdot 2^k-5}$.\end{theorem}
\bigskip
\begin{theorem}\label{thm2'}
Let $\mathbb{F}_p$ be a prime field, and $A \subset\mathbb{F}_p$.  Then for $\E=A^{d} \subset \mathbb{F}_p^{d}$ with $d=2k + 1$, $k\ge 2\in \mathbb{N}$, we have \[\left\vert \Delta(\E)\cdot \Delta(\E)\right\vert=\left\vert\left\lbrace a\cdot b\colon a, b\in \Delta(\E)\right\rbrace\right\vert\gg p,\]
whenever $|\E|\gg p^{\frac{d}{2}-\varepsilon}$ with $\varepsilon=d\cdot \frac{2^{k+1}-5}{10(2^{k}-1)}$.
\end{theorem}
\bigskip

Let us remark that it is not possible to break the exponent $d/2$ for both quotient set and product set of the distance set over arbitrary finite fields. For instance, suppose $q=p^2$, and $\E=A^d\subset \mathbb{F}_q$ with $A=\mathbb{F}_p$. Then we have $|\E|=q^{\frac{d}{2}}$ and $|\Delta(\E)\cdot \Delta(\E)|=|\frac{\Delta(\E)}{\Delta(\E)}|=p=q^{1/2}$.

\section{Proofs of Theorem \ref{thm1} and Theorem \ref{thm2}}
To prove Theorems \ref{thm1} and \ref{thm2}, we make use of the following results. The first result was given by the first author, Vinh and De Zeeuw \cite{pham}. The second was given by Balog \cite{balog}.
\bigskip
\begin{lemma}\label{lm1}
Let $\mathbb{F}_p$ be a prime field, and $A$ be a set in $\mathbb{F}_p$. For $k\ge 2$, we have
\[\left|\Delta(A^k)\right|
\gg  \min\left\{|A|^{2-\frac{1}{2^{k-1}}},p\right\}.\]
\end{lemma}
\bigskip
\begin{lemma}\label{lm2}
Let $\mathbb{F}_q$ be an arbitrary finite field of order $q$, and $B, C$ be sets in $\mathbb{F}_q$. Suppose that $B\cap C=\emptyset$ and $|B||C|\gg q$, then we have
\[\left\vert\frac{B-C}{B-C}\right\vert\ge \frac{q}{3}.\]
\end{lemma}
\bigskip
\begin{lemma}\label{phu1}
Let $\mathbb{F}_p$ be a prime field, and $A$ be a set in $\mathbb{F}_p$. For $k_1, k_2\ge 2$, we have
\[\Delta(A^{k_1+k_2})=\Delta(A^{k_1})+\Delta(A^{k_2}).\]
\end{lemma}
\begin{proof}
We first show that $\Delta(A^{k_1+k_2})\subset \Delta(A^{k_1})+\Delta(A^{k_2})$. Let $t$ be an element in $\Delta(A^{k_1+k_2})$. We now prove that $t$ can be presented as a sum of two elements $t_1\in \Delta(A^{k_1})$ and $t_2 \in \Delta(A^{k_2})$. Indeed, suppose that
\[t=(x_1-y_1)^2+\cdots+(x_{k_1}-y_{k_1})^2+(x_{k_1+1}-y_{k_1+1})^2+\cdots+(x_{k_1+k_2}-y_{k_1+k_2})^2,\]
where $x_i, y_i\in A$. Set $t_1=(x_1-y_1)^2+\cdots+(x_{k_1}-y_{k_1})^2$ and $t_2=(x_{k_1+1}-y_{k_1+1})^2+\cdots+(x_{k_1+k_2}-y_{k_1+k_2})^2$. It is clear that $t_1$ is an element in $\Delta(A^{k_1})$, $t_2$ is an element in $\Delta(A^{k_2})$, and $t=t_1+t_2$. This implies that $\Delta(A^{k_1+k_2})\subset \Delta(A^{k_1})+\Delta(A^{k_2})$.

We now prove the inverse direction $\Delta(A^{k_1})+\Delta(A^{k_2})\subset \Delta(A^{k_1+k_2})$.

Let $t_1$ be an element in $\Delta(A^{k_1})$, $t_2$ be an element in $\Delta(A^{k_2})$. Suppose that $t_1$ is the distance between $x=(x_1,\ldots, x_{k_1})\in A^{k_1}$ and $y=(y_1, \ldots, y_{k_1})\in A^{k_1}$, $t_2$ is the distance between $z=(z_1,\ldots, z_{k_2})\in A^{k_2}$ and $y=(t_1, \ldots, t_{k_2})\in A^{k_2}$. Then we have $t_1+t_2$ is the distance between $(x_1, \ldots, x_{k_1}, z_1\ldots, z_{k_2})\in A^{k_1+k_2}$ and $(y_1, \ldots, y_{k_1}, t_1, \ldots, t_{k_2})\in A^{k_1+k_2}$. Hence, $t_1+t_2\in \Delta(A^{k_1+k_2})$. In other words, $\Delta(A^{k_1})+\Delta(A^{k_2)}\subset \Delta(A^{k_1+k_2})$.
\end{proof}
We are ready to prove Theorem \ref{thm1}.
\paragraph{Proof of Theorem \ref{thm1}:}
Let $X$ be a subset of $\Delta(A^k)$ such that for any $x\in X$ we have $-x\not\in X$. Without loss of generality, we assume that $|X|\ge |\Delta(A^k)|/2$. From Lemma \ref{phu1}, we have $\Delta(\E)=\Delta(A^k)+\Delta(A^k)$. Hence,
\[\left\vert\frac{\Delta(\E)}{\Delta(\E)}\right\vert\ge \left\vert \frac{X-(-X)}{X-(-X)}\right\vert.\]
Set $B=X$ and $C=-X$. It follows from our setting that $B\cap C=\emptyset$. Therefore, applying Lemma \ref{lm2}, we have
\[\left\vert\frac{\Delta(\E)}{\Delta(\E)}\right\vert\ge \frac{p}{3},\]
whenever $|B||C|\gg p$. Since $|B|=|C|=|X|\gg |\Delta(A^k)|$, the condition $|B||C|\gg p$ is equivalent to $|\Delta(A^k)|^2\gg p$. Lemma \ref{lm1} tells us that
\[\left|\Delta(A^k)\right|
\gg  \min\left\{|A|^{2-\frac{1}{2^{k-1}}},p\right\}.\]
Hence, by a direct computation, if $|\E|\gg p^{\frac{d}{2}-\varepsilon}$ with $\varepsilon=\frac{d}{2}\cdot \frac{2^k-2^{k-1}-1}{2^{k}-1}$, then $|A|\gg p^{\frac{2^{k-2}}{2^k-1}}$. So $|\Delta(A^k)|^2\gg p$. This concludes the proof of the theorem.
$\hfill\square$

\paragraph{Proof of Theorem \ref{thm2}:} Let $B$ be a subset of $\Delta(A^k)$ such that $|B|\ge |\Delta(A^k)|/2$ and $B\cap -B=\emptyset$. Let $C$ be a subset of $\Delta(A^{k+1})$ such that $B\subset C$, $C\cap -C=\emptyset$, and $|C|\ge |\Delta(A^{k+1})|/2$.  We note that the condition $B\subset C$ can be satisfied since $\Delta(A^k)\subset \Delta(A^{k+1})$. As in the proof of Theorem \ref{thm1}, we have
\[\left\vert\frac{\Delta(\E)}{\Delta(\E)}\right\vert\ge \left\vert \frac{B-(-C)}{B-(-C)}\right\vert.\]
The condition $B\cap -C=\emptyset$ holds since $B\subset C$ and $C\cap -C=\emptyset$. Lemma \ref{lm2} implies that if $|B||C|\gg p$, then we have
\[\left\vert\frac{\Delta(\E)}{\Delta(\E)}\right\vert \ge \frac{p}{3}.\]
Thus, in the rest of the proof, we will clarify the condition $|B||C|\gg p$. It follows from our setting that $|B||C|\gg |\Delta(A^k)|\cdot |\Delta(A^{k+1})|$. Applying Lemma \ref{lm1}, we get
\[|\Delta(A^k)|\cdot |\Delta(A^{k+1})|\gg \min \left\lbrace p^2, |A|^{\frac{2^{k+2}-3}{2^{k}}}, p|A|^{2-\frac{1}{2^k}}, p|A|^{2-\frac{1}{2^{k-1}}}\right\rbrace.\]
In other words, if $|A|\gg p^{\frac{2^k}{2^{k+2}-3}}$, i.e. $|\E|\gg p^{\frac{d}{2}-\varepsilon}$ with $\varepsilon=d\cdot \frac{2^{k+2}-2^{k+1}-3}{2^{k+3}-6}$, the condition $|B||C|\gg p$ holds. This completes the proof of the theorem. $\hfill\square$
\bigskip

\section{Proofs of Theorem \ref{thm1'} and Theorem \ref{thm2'}}
The ideas in the proofs of Theorems \ref{thm1'} and \ref{thm2'} are similar to those of Theorems \ref{thm1} and \ref{thm2}, except that we will use the following lemma in the place of Lemma \ref{lm2}.
\bigskip
\begin{lemma}[Proof of Theorem F, \cite{m}]\label{mm}
Let $\mathbb{F}_p$ be a prime field of order $p$, and $\A, \B, \C, \D$ be sets in $\mathbb{F}_p$. Let $N(\A, \B, \C, \D)$ be the number of $8$-tuples $(a, b, c, d, a', b', c', d')\in (\A\times \B\times \C\times \D)^2$ such that $(a-b)(c-d)=(a'-b')(c'-d')$. Suppose that $|\A| =|\C|$, $|\B|= |\D|$, and $|\A|\le |\B|$, then we have
\begin{align*}
N(\A, \B, \C, \D)\ll \frac{|\A|^2|\B|^2|\C|^2|\D|^2}{p}+p^{1/2}(|\A||\B||\C||\D|)^{11/8}+\frac{|\A|^{11/4}|\B|^4}{p^{1/4}}+(|\A||\C||\D|)^2.
\end{align*}
\end{lemma}
\bigskip
\paragraph{Proof of Theorem \ref{thm1'}:}
From Lemma \ref{phu1}, we have $\Delta(\E)=\Delta(A^k)+\Delta(A^{k})$. Thus
\[\left\vert \Delta(\E)\cdot \Delta(\E)\right\vert =|\left(\Delta(A^k)+\Delta(A^k)\right)\cdot \left(\Delta(A^k)+\Delta(A^k)\right)|=|(\A-\B)(\C-\D)|,\]
where $\A=\C=\Delta(A^k), ~\B=\D=-\Delta(A^k)$.

By the Cauchy-Schwarz inequality, we have
\begin{equation}\label{tich1}|(\A-\B)(\C-\D)|\ge \frac{|\A|^2|\B|^2|\C|^2|\D|^2}{N(\A, \B, \C, \D)},\end{equation}
where $N(\A, \B, \C, \D)$ is defined as in Lemma \ref{mm}.

Lemma \ref{lm1} gives us  that
\[|\A|=|\B|=|\C|=|\D|\gg \min \left\lbrace |A|^{2-\frac{1}{2^{k-1}}}, p\right\rbrace.\]
Since $|\E|\gg p^{\frac{d}{2}-\varepsilon}$ with $\varepsilon=\frac{d}{2}\cdot \frac{2^{k+1}-5}{5\cdot 2^k-5}$, which is equivalent with $|A|\gg p^{\frac{3\cdot 2^{k-1}}{5\cdot(2^k-1)}}$, we obtain $|\A|=|\B|=|\C|=|\D|\gg p^{3/5}$.  Under this condition and Lemma \ref{mm}, we achieve
\begin{equation}\label{tich2}N(\A, \B, \C, \D)\ll \frac{|\A|^2|\B|^2|\C|^2|\D|^2}{p}.\end{equation}
Putting (\ref{tich1}) and (\ref{tich2}) together, the theorem follows.  $\hfill\square$
\paragraph{Proof of Theorem \ref{thm2'}:} Since $\E=A^{2k+1}$, we have $|\Delta(\E)\cdot \Delta(\E)|\ge |\Delta(A^{2k})\cdot \Delta(A^{2k})|.$ It follows from the proof of Theorem \ref{thm1'} that if $|A|>p^{\frac{3\cdot 2^{k-1}}{5(2^k-1)}}$, then
\[|\Delta(A^{2k})\cdot \Delta(A^{2k})|\gg p.\]
Therefore, under the condition $|\E|\gg p^{\frac{d}{2}-\varepsilon}$ with $\varepsilon=(2k+1) \cdot \frac{2^{k+1}-5}{10(2^k-1)}=d\cdot \frac{2^{k+1}-5}{10(2^k-1)}$, we obtain
\[|\Delta(\E)\cdot \Delta(\E)|\ge |\Delta(A^{2k})\cdot \Delta(A^{2k})|\gg p.\]
This completes the proof of the theorem.  $\hfill\square$

\section{Concluding remarks}

In the setting of arbitrary finite fields $\mathbb{F}_q$, Do and Vinh \cite{hieu} proved that for $A\subset \mathbb{F}_q$ with $|A|\gg q^{1/2}$, we have
\[|\Delta(A^k)|\gg \min \left\lbrace q, \frac{|A|^{2k-1}}{q^{k-1}}\right\rbrace.\]
One can follow the proofs of Theorems \ref{thm1} and \ref{thm2} to show that
\[\left\vert\frac{\Delta(A^{d})}{\Delta(A^{d})}\right\vert, \left\vert\frac{\Delta(A^{d+1})}{\Delta(A^{d+1})}\right\vert\ge  \frac{q}{3},\]
under the condition $|A|\gg q^{1/2}$. This matches Theorem \ref{IKP}.

In the proof of Theorem \ref{thm2'}, one might try to set $\A=\Delta(A^k)=\C, \B=\D=-\Delta(A^{k+1})$. This is clear that $|\A|\le |\B|$. However, in Lemma \ref{mm}, in order to get $N(\A, \B, \C, \D)\ll |\A|^2|\B|^2|\C|^2|\D|^2 p^{-1}$, we need the condition $|\A|>p^{3/5}$. This implies that $|\A|, |\B|, |\C|, |\D|>p^{3/5}$. So we get the same condition on the size of $A$ as in the proof of Theorem \ref{thm1'}. One might also try to apply the bound $|X(Y+Z)|\gg \min \left\lbrace (|X||Y||Z|)^{1/2}, p\right\rbrace$ in \cite{frank} with $X=\Delta(\E), Y=Z=\Delta(A^k)$ or $Y=\Delta(A^k), Z=\Delta(A^{k+1})$ to bound $|\Delta(\E)\cdot \Delta(\E)|$, but the exponents are worse than those of Theorems \ref{thm1'} and \ref{thm2'}.

It is not known if Problem \ref{pr}, the Erd\H os-Falconer distance problem over finite fields, changes over prime fields.  As we mentioned in the introduction, the exponent $(d+1)/2$ can not be improved for odd dimensions over arbitrary finite fields. The constructions in \cite{har}, which demonstrates the sharpness of the exponent $(d+1)/2$, were based on the structures of subfields.  However, in light of our results, one may be able to break this exponent over prime fields.

\end{document}